\documentclass[12pt,a4paper]{amsart}
\usepackage{amssymb,amsxtra}

\calclayout

\newcommand{\+}{\nobreakdash-}
\renewcommand{\:}{\colon}

\newcommand{\rarrow}{\longrightarrow}
\newcommand{\ot}{\otimes}

\DeclareMathOperator{\Max}{Max}
\DeclareMathOperator{\Hom}{Hom}
\DeclareMathOperator{\Ext}{Ext}
\DeclareMathOperator{\Tor}{Tor}
\DeclareMathOperator{\coker}{coker}
\DeclareMathOperator{\pd}{pd}

\newcommand{\Modl}{{\operatorname{\mathsf{--Mod}}}}

\newcommand{\ctra}{{\operatorname{\mathsf{-ctra}}}}

\newcommand{\lrarrow}{\mskip.5\thinmuskip
      \relbar\joinrel\relbar\joinrel\rightarrow\mskip.5\thinmuskip}
\newcommand{\bu}{{\text{\smaller\smaller$\scriptstyle\bullet$}}}

\newcommand{\F}{\mathbb F}
\newcommand{\G}{\mathbb G}

\newcommand{\m}{\mathfrak m}
\newcommand{\n}{\mathfrak n}

\newcommand{\sD}{\mathsf D}

\newcommand{\boZ}{\mathbb Z}
\newcommand{\boQ}{\mathbb Q}

\theoremstyle{plain}
\newtheorem{thm}{Theorem}
\newtheorem{prop}[thm]{Proposition}
\newtheorem{lem}[thm]{Lemma}

\newtheorem{cor}[thm]{Corollary}

\theoremstyle{definition}
\newtheorem{rem}[thm]{Remark}
\newtheorem{dfn}[thm]{Definition}
\newtheorem{setup}[thm]{Setup}

\newcommand{\Section}[1]{\medskip\section{#1}\smallskip}
\begin{document}

\title{Flat commutative ring epimorphisms \\
of almost Krull dimension zero}

\author{Leonid Positselski}

\address{Institute of Mathematics of the Czech Academy of Sciences \\
\v Zitn\'a~25, 115~67 Praha~1 (Czech Republic); and
\newline\indent Laboratory of Algebra and Number Theory \\
Institute for Information Transmission Problems \\
Moscow 127051 (Russia)}

\email{positselski@math.cas.cz}

\begin{abstract}
 We consider flat epimorphisms of commutative rings $R\rarrow U$
such that, for every ideal $I\subset R$ for which $IU=U$,
the quotient ring $R/I$ is semilocal of Krull dimension zero.
 Under these assumptions, we show that the projective dimension of
the $R$\+module $U$ does not exceed~$1$.
 We also describe the Geigle--Lenzing perpendicular subcategory
$U^{\perp_{0,1}}$ in $R\Modl$.
 Assuming additionally that the ring $U$ and all the rings $R/I$ are
perfect, we show that all flat $R$\+modules are $U$\+strongly flat.
 Thus we obtain a generalization of some results of the paper~\cite{BP},
where the case of the localization $U=S^{-1}R$ of the ring $R$
at a multiplicative subset $S\subset R$ was considered.
\end{abstract}

\maketitle

\tableofcontents

\section*{Introduction}
\smallskip

 Let $R$ be a commutative integral domain with the field of
fractions~$Q$.
 It was discovered by Matlis~\cite{Mat} that certain commutative and
homological algebra constructions behave much better when the projective
dimension of the $R$\+module $Q$ does not exceed~$1$.
 Commutative domains $R$ with this property are now known as
\emph{Matlis domains}~\cite{FS}.
 More generally, the same observation applies to the case when $R$ is
a commutative ring and $Q$ is its full ring of quotients, that is
$Q=S^{-1}R$, where $S\subset R$ is the set of all regular elements
in~$R$~\,\cite{Mat2}.

 Even more generally, the case of a multiplicative subset $S\subset R$
consisting of (some) regular elements in $R$ was considered by
Angeleri H\"ugel, Herbera, and Trlifaj in the paper~\cite{AHT}, where
a number of conditions equivalent to $\pd_R S^{-1}R\le 1$ was listed.
 In particular, the projective dimension of the $R$\+module $Q=S^{-1}R$
does not exceed~$1$ if and only if $Q\oplus Q/R$ is a $1$\+tilting
$R$\+module, if and only if two natural notions of $S$\+divisibility
for $R$\+modules coincide, and if and only if $Q/R$ is a direct sum of
countably presented $R$\+modules.
 The next step, towards the study of localizations $S^{-1}R$ where
a multiplicative subset $S\subset R$ is allowed to contain some
zero-divisors in $R$, was made in the present author's
papers~\cite[Section~13]{Pcta} and~\cite{PMat,PSl2,BP}.

 From the contemporary point of view, it appears that the maximal
natural generality in this line of thought is achieved by considering
\emph{ring epimorphisms} $u\:R\rarrow U$.
 In this context, the condition that the projective dimension of
the (left or right) $R$\+module $U$ does not exceed~$1$ continues to
play an important role.
 In particular, for an injective homological ring epimorphism~$u$
(of associative, not necessarily commutative rings) a number of
homologically relevant conditions equivalent to $\pd_R U\le 1$ were 
listed in the paper of Angeleri H\"ugel and
S\'anchez~\cite[Theorem~3.5]{AS}.

 As a general rule, all the above-mentioned results allow to obtain
nice corollaries of the projective dimension~$1$ condition, but none
of them provide a workable technology for \emph{proving}
this condition.
 How does one show that $\pd_R S^{-1}R\le 1$, for a specific
multiplicative subset $S$ in a commutative ring~$R$\,?
 Or more generally that $\pd_R U\le 1$, for a specific ring epimorphism
$u\:R\rarrow U$\,?

 Notice first of all that $S^{-1}R$ is always a \emph{flat} $R$\+module,
while for a ring epimorphism $R\rarrow U$ the (left or right)
$R$\+module $U$ does not have to be flat.
 The following condition is usually imposed on a ring epimorphism,
however: a ring epimorphism $u\:R\rarrow U$ is said to be
\emph{homological} if $\Tor^R_n(U,U)=0$ for all $n\ge1$.
 In the context of commutative rings, the following recent result of
Bazzoni and the present author is relevant in this connection:
\emph{if $R\rarrow U$ is an epimorphism of commutative rings such that\/
$\Tor^R_1(U,U)=0$ and $\pd_RU\le1$, then $U$ is a flat
$R$\+module}~\cite[Theorem~5.2]{BP2}.
 In a sense, this means that one can restrict oneself to flat
epimorphisms when studying the projective dimension~$1$ condition
for epimorphisms of commutative rings.

 Flat ring epimorphisms can be described in terms of \emph{Gabriel
filters} (otherwise known as \emph{Gabriel
topologies})~\cite[Chapters~VI and~IX\+-XI]{St}.
 To every epimorphism of associative rings $u\:R\rarrow U$ such that
$U$ is a flat left $R$\+module, one assigns the set $\G$ of all right
ideals $I\subset R$ such that $R/I\ot_RU=0$, or equivalently, $IU=U$.
 The ring $U$ then can be recovered as the \emph{ring of quotients}
$U=R_\G$ of the ring $R$ with respect to the Gabriel filter~$\G$.
 In particular, if $R$ is commutative and $U=S^{-1}R$, then $\G$ is
the set of all ideals in $R$ intersecting~$S$.

 The simplest class of ring epimorphisms $R\rarrow U$ for which
the projective dimension of the $R$\+module $U$ does not exceed~$1$
is the following one.
 Let $S\subset R$ be a \emph{countable} multiplicative subset.
 Then $\pd_R S^{-1}R\le 1$.
 A far-reaching generalization of this elementary observation was
obtained in the paper~\cite{Pcoun}.
 Specifically, \emph{if $u\:R\rarrow U$ is a left flat ring epimorphism
and the related Gabriel filter of right ideals in $R$ has a countable
base, then the projective dimension of the flat left $R$\+module $U$
does not exceed~$1$} \cite[Theorem~8.5]{Pcoun}.
 The proof of this result given in~\cite{Pcoun} uses the fundamental
concept of the abelian category of \emph{contramodules} over
a topological ring (specifically, over the completion of $R$
with respect to~$\G$).

 Another basic observation is that $\pd_R S^{-1}R\le 1$ whenever $R$
is a Noetherian commutative ring of Krull dimension~$1$.
 In fact, any flat module over a Noetherian commutative ring of Krull
dimension~$\le1$ has projective dimension~$\le1$
\,\cite[Corollaire~II.3.2.7]{RG}.
 A contramodule-based proof of this assertion was suggested
in~\cite[Corollary~13.7 and Remark~13.10]{Pcta}.
 Building up on the techniques of~\cite{Pcta}, Bazzoni and the present
author proved the following result in~\cite[Theorem~6.13]{BP}:
\emph{for any commutative ring $R$ with a multiplicative subset
$S\subset R$ such that $R/sR$ is a semilocal ring of Krull dimension
zero for every $s\in S$, one has} $\pd_R S^{-1}R\le1$.

 In this paper, we generalize some results of~\cite{BP} to the case of
a flat epimorphism of commutative rings $u\:R\rarrow U$.
 Let $\G$ denote the Gabriel filter of ideals in $R$ related to~$u$.
 We show that $\pd_RU\le1$ \emph{whenever the quotient ring $R/I$ is
semilocal of Krull dimension zero for every ideal $I\in\G$}.
 The argument is based on the notion of \emph{$I$\+contramodule
$R$\+modules} for an ideal $I$ in a commutative ring~$R$.
 This result of ours has already found its uses in the work of Bazzoni
and Le~Gros on envelopes and covers in the tilting cotorsion pairs
related to $1$\+tilting modules over commutative rings;
see~\cite[Remark~8.6 and Theorem~8.7]{BG1} and~\cite[Remark~8.2 and
Theorem~8.17]{BG2}.

 The proofs of the $\pd_RU\le1$ theorems in the papers~\cite{Pcoun},
\cite{BP}, and the present paper follow a single common approach,
which can be described in very general terms as follows.
 Given a ring epimorphism $u\:R\rarrow U$, one considers the two-term
complex $K_u^\bu=(R\to U)$.
 When $R$ and $U$ are associative rings, $K_u^\bu$ is a complex of
$R$\+$R$\+bimodules; in the commutative case, it is simply a complex
of $R$\+modules.
 To any (left) $R$\+module $B$, one assigns the sequence of (left)
$R$\+modules $\Ext^n_R(K_u^\bu,B)=\Hom_{\sD(R\Modl)}(K_u^\bu,B[n])$
of homomorphisms in the derived category $\sD(R\Modl)$, indexed by
the integers~$n\ge0$.
 In particular, for $n\ge2$ there is a natural $R$\+module isomorphism
$\Ext^n_R(K_u^\bu,B)\simeq\Ext^n_R(U,B)$ \cite[Lemma~2.1(b)]{BP2}.

 In each of the three settings mentioned in the italicized phrazes
in the previous paragraphs where the results of the papers~\cite{Pcoun},
\cite{BP}, and the present paper are discussed, the argument concerning 
the projective dimension $\pd_RU$ proceeds as follows.
 In order to prove that $\Ext^n_R(U,B)=0$ for $n\ge2$, one shows that
the two classes of $R$\+modules of the form $\Ext^n_R(K_u^\bu,B)$ and
$\Ext^n_R(U,B)$ (where $n\ge\nobreak0$ and $B\in R\Modl$) only
intersect at zero.
 The two kinds of Ext modules just have incompatible properties.
 In fact, $\Ext^n_R(U,B)$ is the underlying $R$\+module of
a $U$\+module.
 The $R$\+modules $\Ext^n_R(K_u^\bu,B)$ are described in terms of some
kind of contramodules (it depends on the setting which specific
contramodule category is more convenient to work with).
 One wants to show that \emph{no} nonzero $R$\+module of the form
$C=\Ext^n_R(K_u^\bu,B)$ admits an extension of its $R$\+module
structure to a $U$\+module structure.
 In fact, one usually proves that $\Hom_R(U,C)=0$; this is certainly
enough.

 The \emph{Geigle--Lenzing perpendicular subcategory} $U^{\perp_{0,1}}$
to the $R$\+module $U$ in the category of $R$\+modules $R\Modl$
consists all $R$\+modules $C$ such that $\Hom_R(U,C)=0=\Ext^1_R(U,C)$
\,\cite{GL}.
 We use the name \emph{$u$\+contramodules} for $R$\+modules
$C\in U^{\perp_{0,1}}$, and the notation $R\Modl_{u\ctra}=
U^{\perp_{0,1}}\subset R\Modl$ for the full subcategory formed by
such modules.
 Another result of this paper is a description of the abelian category
$R\Modl_{u\ctra}$ in terms of the abelian categories of
$\m$\+contramodule $R$\+modules, where $\m$~ranges over the maximal
ideals of $R$ belonging to~$\G$.
 The assumption that $R/I$ is semilocal of Krull dimension zero
for all $I\in\G$ is made here.

 Finally, let us recall that a module $F$ over a commutative domain~$R$
is said to be \emph{strongly flat} if it is a direct summand of
an $R$\+module $G$ appearing in a short exact sequence of $R$\+modules
$0\rarrow V\rarrow G\rarrow W\rarrow 0$ in which $V$ is a free
$R$\+module and $W$ is a $Q$\+vector space~\cite{Trl,BS}.
 A series of generalizations of this concept was developed in
the papers~\cite{FS2,BP,PSl2,Pcoun}.
 We refer to the introduction to~\cite{PSl2} for a detailed discussion
with further references.

 In particular, let $R$ be a commutative ring and $S\subset R$
be a multiplicative subset.
 Then an $R$\+module $F$ is called \emph{$S$\+strongly flat}~\cite{BP}
if it is a direct summand of an $R$\+module $G$ appearing in a short
exact sequence of $R$\+modules $0\rarrow V\rarrow G\rarrow W\rarrow 0$
in which $V$ is a free $R$\+module and $W$ is a free $S^{-1}R$\+module.
 In the terminology of~\cite{BP}, a ring $R$ is said to be
\emph{$S$\+almost perfect} if $S^{-1}R$ is a perfect ring and
$R/sR$ is a perfect ring for every $s\in S$.
 According to~\cite[Theorem~7.9]{BP}, \,$R$ is $S$\+almost perfect if
and only if all flat $R$\+modules are $S$\+strongly flat.

 More generally, given a left flat epimorphism of associative rings
$u\:R\rarrow U$, a left $R$\+module $F$ is called \emph{$U$\+strongly
flat}~\cite[Section~9]{Pcoun} if it is a direct summand of
a left $R$\+module $G$ appearing in a short exact sequnce of left
$R$\+modules $0\rarrow V\rarrow G\rarrow W\rarrow 0$ in which $V$ is
a free left $R$\+module and $W$ is a free left $U$\+module.

 Let $u\:R\rarrow U$ be a flat epimorphism of commutative rings and
$\G$ be the related Gabriel filter of ideals in~$R$.
 In the terminology of~\cite[Sections~6 and~8]{BG2}, a ring $R$ is
said to be \emph{$\G$\+almost perfect} if $U=R_\G$ is a perfect ring
and $R/I$ is a perfect ring for all $I\in\G$.
 We show that all flat $R$\+modules are $U$\+strongly flat whenever $R$
is $\G$\+almost perfect (the converse assertion also holds).
 More generally, in the assumption that $R/I$ is semilocal of Krull
dimension zero for all $I\in\G$, we characterize $U$\+strongly flat
$R$\+modules $F$ as flat $R$\+modules for which the $U$\+module
$U\ot_RF$ is projective and the $R/I$\+modules $F/IF$ are projective.

\subsection*{Acknowledgement}
 I~am grateful to Silvana Bazzoni and Michal Hrbek for numerous very 
helpful discussions and communications.
 The author is supported by the GA\v CR project 20-13778S and
research plan RVO:~67985840.

\Section{$\F$-h-locality} \label{h-locality-secn}

 Recall that a commutative ring $T$ is said to be \emph{semilocal} if
the set of all maximal ideals in $T$ is finite.
  The ring $T$ is said to have \emph{Krull dimension zero} if all
the prime ideals in $T$ are maximal.
 The following description is standard and easy.

\begin{lem} \label{semilocal-nil-lemma}
\textup{(a)} A commutative ring $T$ is isomorphic to a finite product of local rings if and only if $T$ is semilocal with the additional
property that every prime ideal of $T$ is contained in a unique
maximal ideal. \par
\textup{(b)} A commutative ring $T$ is semilocal of Krull dimension zero
if and only if it is isomorphic to a finite product of local rings
$T_1\times\dotsb\times T_m$ such that, for every $1\le j\le m$,
the maximal ideal in $T_j$ consists of nilpotent elements. \qed
\end{lem}

 A \emph{filter\/ $\F$ of ideals} (or a ``linear topology'') in
a commutative ring $R$ is a set of ideals in $R$ such that $R\in\F$,
and $K\supset I\cap J$, \ $I$, $J\in\F$ implies $K\in\F$.

 Let $R$ be a commutative ring with a filter $\F$ of ideals.
 An $R$\+module $N$ is said to be \emph{$\F$\+torsion} if, for every
element $x\in N$, the annihilator of~$x$ in $R$ belongs to~$\F$.

 We will say that the ring $R$ is \emph{$\F$\+h-local}
\cite[Section~IV.3]{FS}, \cite[Section~7]{BG2} if every ideal $I\in\F$
is contained only in finitely many maximal ideals of $R$ and every prime
ideal of $R$ belonging to $\F$ is contained in a unique maximal ideal.
 In other words, $R$ is $\F$\+h-local if and only if for every $I\in\F$
the quotient ring $R/I$ is semilocal and every prime ideal of $R/I$
is contained in a unique maximal ideal.
 Equivalently, this means that the ring $R/I$ is a finite product of
local rings (by Lemma~\ref{semilocal-nil-lemma}(a)).

 As usually, we denote by $\Max R$ the set of all maximal ideals of
the ring~$R$.
 Given a maximal ideal $\m\subset R$ and an $R$\+module $N$,
one can consider the local ring $R_\m$ and the $R_\m$\+module~$N_\m$,
that is, the localizations of $R$ and $N$ at~$\m$.
 Conversely, any $R_\m$\+module can be considered as an $R$\+module.
 Notice that if $N$ is $\F$\+torsion and $\m\notin\F$, then $N_\m=0$
(indeed, for any $x\in N$, the annihilator of~$x$ in $R$ is
not contained in~$\m$).

\begin{lem} \label{h-local-filter}
 Let $R$ be a commutative ring with a filter\/ $\F$ of ideals such
that the ring $R$ is\/ $\F$\+h-local.
 Given an\/ $\F$\+torsion $R$\+module $N$, consider all the maximal
ideals\/~$\m$ of the ring $R$ such that\/ $\m\in\F$.
 Then the natural map $N\rarrow\prod_\m N_\m$ whose components are
the localization maps $N\rarrow N_\m$ factorizes through the submodule\/
$\bigoplus_\m N_\m\subset\prod_\m N_\m$ and induces an isomorphism
$N\simeq\bigoplus_\m N_\m$.

 Conversely, if\/ $\F$ is a filter of ideals in $R$ such that for
every\/ $\F$\+torsion $R$\+module $N$ there is an isomorphism of
$R$\+modules $N\simeq\bigoplus_\m N_\m$, then the ring $R$ is\/
$\F$\+h-local.
\end{lem}

\begin{proof}
 This is~\cite[Proposition~7.4]{BG2}.
\end{proof}

 It follows from Lemma~\ref{h-local-filter} that, for any $\F$\+h-local
ring $R$, the category of $\F$\+torsion $R$\+modules is equivalent to
the Cartesian product of the categories of $\F$\+torsion
$R_\m$\+modules, taken over the maximal ideals~$\m$ of the ring~$R$
belonging to the filter~$\F$.
 The direct sum functor
$$
 (N(\m))_{\m\in\F\cap\Max R}\,\longmapsto\,
 \bigoplus_{\m\in\F\cap\Max R} N(\m)
$$
establishes the equivalence, with an inverse equivalence provided by
the localization functor
$$
 N\,\longmapsto\,(N_\m)_{\m\in\F\cap\Max R}.
$$
 In particular, the decomposition of an $R$\+module $N$ into a direct
sum of $\F$\+torsion $R_\m$\+modules $N(\m)$ is unique and functorial
if it exists (and it exists if and only if $N$ is $\F$\+torsion).
 All these assertions can be found in the discussion
in~\cite[Section~7]{BG2}.
 Alternatively, they can be deduced directly from
Lemma~\ref{h-local-filter} using the dual version
of~\cite[Lemma~8.6]{BP2}.

 Let $R$ be a commutative ring with a filter $\F$ of ideals.
 We will say that $R$ is \emph{$\F$\+h-nil} \cite[Section~6]{BP},
\cite[Section~7]{BG2} if every ideal $I\in\F$ is contained only in
finitely many maximal ideals of $R$ and all the prime ideals of $R$
belonging to $\F$ are maximal.
 In other words, $R$ is $\F$\+h-nil if and only if for every $I\in\F$
the quotient ring $R/I$ is semilocal of Krull dimension zero.
 Equivalently, this means that $R/I$ is a finite product of local rings
whose maximal ideals consist of nilpotent elements
(by Lemma~\ref{semilocal-nil-lemma}(b)).
 Obviously, any $\F$\+h-nil ring is $\F$\+h-local.

 We refer to the book~\cite[Chapter~VI]{St} for the definition of
a \emph{Gabriel filter of ideals} (also known as a Gabriel topology)
in a ring~$R$.
 A filter $\F$ of ideals in a ring $R$ is said to have a \emph{base of
finitely generated ideals} if for any ideal $J\in\F$ there exists
a finitely generated ideal $I\subset R$ such that $I\in\F$
and $I\subset J$.

 In the rest of this section (and partly in the next one)
we will be working in the following setup.

\begin{setup} \label{nilsetup}
 Let $R$~be a commutative ring and $\G$ be a Gabriel filter of ideals
in $R$ with a base of finitely generated ideals.
 We assume that the ring $R$ is $\G$\+h-nil.
\end{setup}

 Given a commutative ring $R$ and an ideal $I\subset R$, we say that
an $R$\+module $N$ is \emph{$I$\+torsion} if for every $s\in I$ and
$x\in N$ there exists an integer $n\ge1$ such that $s^nx=0$ in~$N$.

\begin{lem} \label{local-G-torsion-m-torsion}
 Assume Setup~\ref{nilsetup}, and let\/ $\m$~be a maximal ideal of $R$
belonging to\/~$\G$.
 Then an $R$\+module $N$ is a\/ $\G$\+torsion $R_\m$\+module if and only
if $N$ is an\/ $\m$\+torsion $R$\+module.
\end{lem}

\begin{proof}
 ``If'': clearly, any $\m$\+torsion $R$\+module $N$ is
an $R_\m$\+module.
 To show that $N$ is $\G$\+torsion, choose a finitely generated ideal
$I$ such that $I\subset\m$ and $I\in\G$.
 Then for every $x\in N$ there exists $n\ge1$ such that $I^nx=0$ in~$N$.
 Since $I^n\in\G$ \cite[Lemma~VI.5.3]{St}, it follows that
the annihilator of~$x$ in $R$ belongs to~$\G$.

 ``Only if'': let $N$ be a $\G$\+torsion $R_\m$\+module and $x\in N$
be an element.
 Choose an ideal $I\in\G$ annihilating~$x$.
 Then the cyclic submodule $R_\m x\subset N$ is a module over the ring
$R_\m/R_\m I =(R/I)_\m$.
 The ring $(R/I)_\m$ is local and, by Setup~\ref{nilsetup} and
Lemma~\ref{semilocal-nil-lemma}(b), its maximal ideal $R_\m\m/R_\m I$
consists of nilpotent elements.
 Hence every module over $(R/I)_\m$ is $\m$\+torsion. 
\end{proof}

\begin{cor} \label{G-torsion-m-torsion-coprod}
 Assuming Setup~\ref{nilsetup}, an $R$\+module $N$ is\/ $\G$\+torsion
if and only if it is isomorphic to a direct sum of some\/ $\m$\+torsion
$R$\+modules $N(\m)$ over the maximal ideals\/ $\m\in\G\cap\Max R$,
$$
 N\simeq\bigoplus_{\m\in\G\cap\Max R}N(\m).
$$
 Such a direct sum decomposition is unique and functorial when it
exists, and the $R$\+modules $N(\m)$ can be recovered as
the localizations $N(\m)=N_\m$.
\end{cor}

\begin{proof}
 The ring $R$ in Setup~\ref{nilsetup} is $\G$\+h-local (as $\G$\+h-nil
implies $\G$\+h-local).
 Hence the assertion follows from Lemmas~\ref{h-local-filter}
and~\ref{local-G-torsion-m-torsion}.
\end{proof}

\Section{$I$-Contramodule $R$-Modules}

 Let $R$ be a commutative ring.
 Given an element $t\in R$, an $R$\+module $C$ is said to be
a \emph{$t$\+contramodule} if $\Hom_R(R[t^{-1}],C)=0=
\Ext^1_R(R[t^{-1}],C)$.
 Given an ideal $I\subset R$, an $R$\+module $C$ is said to be
an \emph{$I$\+contramodule} if $C$ is a $t$\+contramodule for
every $t\in I$.
 The class of all $I$\+contramodule $R$\+modules is closed under
the kernels, cokernels, extensions, and infinite products in $R\Modl$
\cite[Lemma~5.1(1)]{BP}.
 The full subcategory of $I$\+contramodule $R$\+modules is denoted
by $R\Modl_{I\ctra}\subset R\Modl$.

\begin{lem} \label{maximal-ideal-contramodule}
 Let $R$ be a commutative ring and\/ $\m\subset R$ be a maximal ideal.
 Then any\/ $\m$\+contramodule $R$\+module is an $R_\m$\+module.
\end{lem}

\begin{proof}
 This is~\cite[Lemma~5.1(2)]{BP}.
\end{proof}

\begin{lem} \label{ext-from-torsion-module-decomp}
 Assuming Setup~\ref{nilsetup}, for any\/ $\G$\+torsion $R$\+module $N$,
any $R$\+module $B$, and any integer~$n\ge0$, the $R$\+module
$P=\Ext^n_R(N,B)$ can be presented as a product of some\/
$\m$\+contramodule $R$\+modules $P(\m)$ over the maximal ideals\/
$\m\in\G\cap\Max R$,
$$
 \Ext^n_R(N,B)\,\simeq\,\prod_{\m\in\G\cap\Max R} P(\m).
$$
\end{lem}

\begin{proof}
 Follows from Corollary~\ref{G-torsion-m-torsion-coprod}
and~\cite[Lemma~6.5(1)]{BP}.
\end{proof}

\begin{lem} \label{ext-from-invertible-to-contra}
 Let $R$ be a commutative ring, $t\in R$ be an element, $D$ be
an $R[t^{-1}]$\+module, and $Q$ be a $t$\+contramodule $R$\+module.
 Then\/ $\Ext^i_R(D,Q)=0$ for all integers\/ $i\ge0$.
\end{lem}

\begin{proof}
 The element~$t$ acts invertibly in the $R$\+module $E=\Ext^i_R(D,Q)$,
since it acts invertibly in~$D$.
 On the other hand, the $R$\+module $E$ is a $t$\+contramodule, since
the $R$\+module $Q$ is~\cite[Lemma~6.5(2)]{BP}.
 Thus $E\simeq\Hom_R(R[t^{-1}],E)=0$.
\end{proof}

 Given a filter $\F$ of ideals in a commutative ring $R$, we will say
that the ring $R$ is \emph{$\F$\+h-semilocal} if the ring $R/I$ is
semilocal for all $I\in\F$.
 This captures ``a half of'' the conditions from the definition of
an $\F$\+h-local ring in Section~\ref{h-locality-secn}.

\begin{prop} \label{ext-between-products}
 Let\/ $\F$ be a filter of ideals in a commutative ring $R$ such that\/
$\F$ has a base of finitely generated ideals and the ring $R$ is\/
$\F$\+h-semilocal.
 Let $P(\m)$ and $Q(\m)$ be two families of\/ $\m$\+contramodule
$R$\+modules, indexed over the maximal ideals\/ $\m\in\F\cap\Max R$.
 Then, for every $i\ge0$, the natural map
$$
 \prod_{\m\in\F\cap\Max R}\Ext^i_R\bigl(P(\m),Q(\m)\bigr)\lrarrow
 \Ext^i_R\left(\prod_{\m\in\F\cap\Max R}P(\m),\>
 \prod_{\m\in\F\cap\Max R}Q(\m)\right)
$$
is an isomorphism.
\end{prop}

\begin{proof}
 Fix $\n\in\F\cap\Max R$, and denote by $P$ the product of $P(\m)$ over
all $\m\in\F\cap\Max R$, \ $\m\ne\n$.
 We have to show that $\Ext^i_R(P,Q(\n))=0$.
 
 Let $I$ be a finitely generated ideal in $R$ such that $I\subset\n$ and
$I\in\F$.
 Since $R$ is $\F$\+h-semilocal, the set of all maximal ideals~$\m$ of
the ring $R$ containing the ideal $I$ is finite.

 Let $s_1$,~\dots, $s_m$ be a finite set of generators of the ideal~$I$.
 For any ideal $\m\in\F\cap\Max R$ such that $\m$~does not contain $I$,
there exists an integer $1\le j\le m$ such that $\m$~does not
contain~$s_j$.
 Denote by $P_j$ the product of the $R$\+modules $P(\m)$ over all
the maximal ideals $\m\in\F\cap\Max R$ such that $s_j\notin\m$.

 By Lemma~\ref{maximal-ideal-contramodule}, \,$P(\m)$ is
an $R_\m$\+module.
 Hence $s_j$~acts invertibly in~$P_j$.
 On the other hand, $Q(\n)$ is an $s_j$\+contramodule.
 By Lemma~\ref{ext-from-invertible-to-contra}, it follows that
$\Ext^i_R(P_j,Q(\n))=0$.

 It remains to show that $\Ext^i_R(P(\m),Q(\n))=0$ for each one among
the finite set of maximal ideals $\m\in\F\cap\Max R$ such that
$I\subset\m$, \,$\m\ne\n$.
 Let $s\in R$ be an element such that $s\in\n$ and $s\notin\m$.
 Then $s$~acts invertibly in $P(\m)$, while $Q(\n)$ is
an $s$\+contramodule.
 Once again, it follows that $\Ext^i_R(P(\m),Q(\n))=0$.
\end{proof}

\begin{cor} \label{full-subcategory-of-products}
 For any filter\/ $\F$ of ideals in a commutative ring $R$ such that\/
$\F$ has a base of finitely generated ideals and the ring $R$ is\/
$\F$\+h-semilocal, the full subcategory of all $R$\+modules of
the form\/ $\prod_{\m\in\F\cap\Max R}C(\m)$, where $C(\m)$ are\/
$\m$\+contramodule $R$\+modules, is closed under the kernels, cokernels,
extensions, and infinite products in $R\Modl$.
\end{cor}

\begin{proof}
 Follows from Proposition~\ref{ext-between-products}
(for $i=0$ and~$1$).
\end{proof}

 Given a complex of $R$\+modules $M^\bu$, an $R$\+module $B$, and
an integer~$n$, we will denote by $\Ext^n_R(M^\bu,B)$ the derived
category Hom module $\Hom_{\sD(R\Modl)}(M^\bu,B[n])$.
 Here $\sD(R\Modl)$ is the derived category of $R$\+modules.

 The following proposition is the key technical assertion of this
paper, on which the proofs of the main results are based.

\begin{prop} \label{ext-from-complex-decomp}
 Assume Setup~\ref{nilsetup}.
 Let $M^\bu=(M^{-1}\to M^0)$ be a two-term complex of $R$\+modules
whose cohomology modules $H^{-1}(M^\bu)$ and $H^0(M^\bu)$ are\/
$\G$\+torsion.
 Then for any $R$\+module $B$ and any integer $n\ge0$,
the $R$\+module\/ $C=\Ext^n_R(M^\bu,B)$ can be presented as
the product of some\/ $\m$\+contramodule $R$\+modules $C(\m)$ over
the maximal ideals\/ $\m\in\G\cap\Max R$,
$$
 \Ext^n_R(M^\bu,B)\,\simeq\,\prod_{\m\in\G\cap\Max R} C(\m).
$$
\end{prop}

\begin{proof}
 By Lemma~\ref{ext-from-torsion-module-decomp}, the $R$\+modules
$\Ext^n_R(H^{-1}(M^\bu),B)$ and $\Ext^n_R(H^0(M^\bu),B)$ can be
decomposed as direct products of $\m$\+contramodule $R$\+modules
for all $n\ge0$.
 The existence of a similar direct product decompositions of
the $R$\+modules $\Ext^n_R(M^\bu,B)$ follows from the natural
long exact sequence of $R$\+modules
\begin{multline*}
 \dotsb\lrarrow\Ext^{n-2}(H^{-1}(M^\bu),B)
 \lrarrow\Ext^n_R(H^0(M^\bu),B) \\
 \lrarrow\Ext^n_R(M^\bu,B)\lrarrow \\
 \Ext_R^{n-1}(H^{-1}(M^\bu),B)\lrarrow
 \Ext^{n+1}(H^0(M^\bu),B)\lrarrow\dotsb 
\end{multline*} 
in view of Corollary~\ref{full-subcategory-of-products}
(cf.~\cite[proof of Lemma~6.11]{BP}
or~\cite[proof of Lemma~8.2]{Pcoun}).
\end{proof}

 An obvious generalization of Proposition~\ref{ext-from-complex-decomp}
holds for any bounded above complex of $R$\+modules $M^\bu$ with
$\G$\+torsion cohomology modules and any bounded below complex of
$R$\+modules $B^\bu$, as one can see from the related spectral sequence.
 We do not include the details, as we do not need them.

\Section{Projective Dimension~$1$ Theorem} \label{projdim1-secn}

 Let $u\:R\rarrow U$ be a flat epimorphism of commutative rings.
 This means that $u$~is a homomorphism of commutative rings such that
$U$ is a flat $R$\+module and the induced ring homomorphisms
$U\rightrightarrows U\ot_R U\rarrow U$ are isomorphisms.

 Denote by $\G$ the set of all ideals $I\subset R$ such that
$R/I\ot_RU=0$.
 Then $\G$ is a Gabriel filter of ideals in $R$, and the ring $U$
can be recovered as the ring of quotients $U=R_\G$ of the ring $R$
with respect to the Gabriel filter~$\G$.

 A Gabriel filter is said to be \emph{perfect} if it corresponds to
a flat epimorphism of rings in this way~\cite[Chapter~XI]{St}.
 Notice that any perfect Gabriel filter $\G$ has a base of finitely
generated ideals~\cite[Section~XI.3]{St}.

 The main results of this paper presume the following setup.

\begin{setup} \label{thesetup}
 Let $u\:R\rarrow U$ be a flat epimorphism of commutative rings,
and let $\G$ be the related perfect Gabriel filter of ideals in~$R$.
 We assume that, for every $I\in\G$, the quotient ring $R/I$ is
semilocal of Krull dimension zero (in other words, the ring $R$
is $\G$\+h-nil in our terminology).
\end{setup}

 Following~\cite[Section~13]{Pcta}, \cite[Section~1]{PMat},
\cite[Section~3]{PSl2}, \cite[Section~4]{BP},
\cite[Section~8]{Pcoun}, \cite[Section~2]{BP2}, etc.,
we will denote by $K_u^\bu$ the two-term complex of $R$\+modules
$R\rarrow U$, with the term $R$ placed in the cohomological degree~$-1$
and the term $U$ placed in the cohomological degree~$0$.
 The next corollary is the particular case of
Proposition~\ref{ext-from-complex-decomp} which we will use.

\begin{cor} \label{ext-from-K-decomp}
 Assuming Setup~\ref{thesetup}, for any $R$\+module $B$ and any integer
$n\ge0$, the $R$\+module\/ $C=\Ext^n_R(K_u^\bu,B)$ can be presented as
the product of some\/ $\m$\+contramodule $R$\+modules $C(\m)$ over
the maximal ideals\/ $\m\in\G\cap\Max R$,
$$
 \Ext^n_R(K_u^\bu,B)\,\simeq\,\prod_{\m\in\G\cap\Max R} C(\m).
$$
\end{cor}

\begin{proof}
 The complex of $U$\+modules $U\ot_RK_u^\bu$ is contractible, hence
the cohomology modules $H^{-1}(K_u^\bu)$ and $H^0(K_u^\bu)$ of
the complex $K_u^\bu$ are $\G$\+torsion $R$\+modules.
 Therefore, Proposition~\ref{ext-from-complex-decomp} is applicable.
\end{proof}

 The aim of this section is to prove the following theorem, which
is our most important result.

\begin{thm} \label{projdim1-theorem}
 Assuming Setup~\ref{thesetup}, the projective dimension of
the $R$\+module $U$ cannot exceed\/~$1$.
\end{thm}

 The argument is based on Corollary~\ref{ext-from-K-decomp}.
 We need several more lemmas before proceeding to prove the theorem.

\begin{lem} \label{one-element-weak-nakayama}
 Let $R$ be a commutative ring, $t\in R$ be an element, and $C$ be
a $t$\+contramodule $R$\+module.
 Then $C/tC\ne0$ whenever $C\ne0$.
\end{lem}

\begin{proof}
 Assume that $C=tC$ for an $R$\+module~$C$.
 Then one easily observes that the natural $R$\+module map
$\Hom_R(R[t^{-1}],C)\rarrow C$ is surjective.
 If $C$ is a $t$\+contramodule, then $\Hom_R(R[t^{-1}],C)=0$,
and it follows that $C=0$.
\end{proof}

\begin{lem} \label{fin-gen-ideal-weak-nakayama}
 Let $R$ be a commutative ring and $I=(s_1,\dotsc,s_m)\subset R$ be
a finitely generated ideal.
 Let $C$ be an $I$\+contramodule $R$\+module.
 Then $C/IC\ne0$ whenever $C\ne0$.
\end{lem}

\begin{proof}
 Assume that $C\ne0$.
 Then, by Lemma~\ref{one-element-weak-nakayama}, we have $C/s_1C\ne0$.
 Furthermore, $C/s_1C$ is still an $I$\+contramodule, since it is
the cokernel of the $R$\+module morphism $s_1\:C\rarrow C$ between two
$I$\+contramodules.
 Applying Lemma~\ref{one-element-weak-nakayama} again, we see that
$C/(s_1C+s_2C)\ne0$, etc.
\end{proof}

\begin{lem} \label{div-to-contra-epi-lemma}
 Let $u\:R\rarrow U$ be a flat epimorphism of commutative rings, and
let\/ $\G$ be the related perfect Gabriel filter of ideals in~$R$.
 Suppose that $D$ is a $U$\+module and $C$ is an $I$\+contramodule
$R$\+module, where $I\in\G$.
 Suppose further that there is a surjective $R$\+module morphism
$D\rarrow C$.
 Then $C=0$.
\end{lem}

\begin{proof}
 Without loss of generality we can assume the ideal $I\in\G$ to be
finitely generated (because the filter $\G$ has a base of finitely
generated ideals).
 We have $D/ID=0$, since $R/I\ot_RU=0$.
 As the morphism $D/ID\rarrow C/IC$ is surjective, it follows that
$C/IC=0$.
 By Lemma~\ref{fin-gen-ideal-weak-nakayama}, we can conclude that
$C=0$.
\end{proof}

\begin{proof}[Proof of Theorem~\ref{projdim1-theorem}]
 Let $B$ be an $R$\+module.
 Then, for every integer $n\ge2$, there is a natural isomorphism of
$R$\+modules $\Ext^n_R(K_u^\bu,B)\simeq\Ext^n_R(U,B)$
(see~\cite[Lemma~2.1(b)]{BP2}; cf.~\cite[Lemma~4.8(3)]{BP}).
 By Corollary~\ref{ext-from-K-decomp}, we have $\Ext^n_R(K_u^\bu,B)
\simeq\prod_{\m\in\G\cap\Max R}C(\m)$, where $C(\m)$ are some
$\m$\+contramodule $R$\+modules.
 On the other hand, $D=\Ext^n_R(U,B)$ is a $U$\+module.
 For every~$\m$, the $R$\+module $C(\m)$ is a quotient (in fact,
a direct summand) of the $R$\+module~$D$.
 By Lemma~\ref{div-to-contra-epi-lemma}, it follows that $C(\m)=0$,
hence $\Ext^n_R(U,B)\simeq\Ext^n_R(K_u^\bu,B)=0$.
\end{proof}

\Section{Divisible Modules} \label{divisible-secn}

 The following definitions, generalizing the classical notions of
$S$\+divisible and $S$\+h-divisible modules for a multiplicative
subset $S\subset R$ and the related ring epimorphism
$R\rarrow S^{-1}R$ \,\cite[Section~1]{PMat}, are quite natural.

\begin{dfn}
 Given a filter $\F$ of ideals in a commutative ring $R$
and an $R$\+module $D$, we say that $D$ is
\emph{$\F$\+divisible}~\cite[Section~VI.9]{St}
if $R/I\ot_RD=0$ for all $I\in\F$.
\end{dfn}

\begin{dfn}
 Let $u\:R\rarrow U$ be an epimorphism of commutative rings.
 We say that an $R$\+module $D$ is \emph{$u$\+divisible}
(or \emph{$u$\+h-divisible}) \cite[Remark~1.2(1)]{BP2} if $D$ is
a quotient $R$\+module of a $U$\+module.
\end{dfn}

 For a flat epimorphism~$u$ and the related Gabriel filter $\G$,
one can immediately see that any $u$\+divisible $R$\+module is
$\G$\+divisible.
 The converse is not true in general, even when $U=S^{-1}R$
\,\cite[Proposition~6.4]{AHT}, \cite[Lemma~1.8(b)]{PMat}.

\begin{thm} \label{divisible-theorem}
 Let $u\:R\rarrow U$ be a flat epimorphism of commutative rings, and
let\/ $\G$ be the related perfect Gabriel filter of ideals in~$R$.
 Assume that the projective dimension of the $R$\+module $U$ does not
exceed~$1$.
 Then the classes of $u$\+divisible and\/ $\G$\+divisible $R$\+modules
coincide.
\end{thm}

\begin{proof}
 The following argument based on the results of the paper~\cite{AH}
was communicated to the author by S.~Bazzoni.
 Denote the $R$\+module $\coker(u)=U/u(R)$ simply by~$U/R$.
 Then, by~\cite[Example~6.5]{MS}, the direct sum $U\oplus U/R$
is a silting $R$\+module.
 The related silting class is the class of all quotient modules of
direct sums of copies of $U\oplus U/R$, which means the class of all
$u$\+divisible $R$\+modules.
 By~\cite[Theorem~4.7]{AH}, for any silting class over a commutative
ring $R$ there exists a Gabriel filter $\G$ (with a base of finitely
generated ideals) in $R$ such that the silting class consists of
all the $\G$\+divisible $R$\+modules.
 So the class of all $u$\+divisible $R$\+modules coincides with
the class of all $\G$\+divisible $R$\+modules for some Gabriel
filter~$\G$.
 It follows immediately that $\G$ is the perfect Gabriel filter
related to~$u$, as desired.
\end{proof}

\begin{rem}
 For \emph{injective} flat epimorphisms of commutative rings
$u\:R\rarrow U$ (i.~e., when the map~$u$ is injective), the following
converse assertion to Theorem~\ref{divisible-theorem} holds.
 If all $\G$\+divisible $R$\+modules are $u$\+divisible, then
the projective dimension of the $R$\+module $U$ does not exceed~$1$.
 This is a part of~\cite[Theorem~5.4]{Hrb}.

 However, there do exist noninjective flat commutative ring
epimorphisms~$u$ of projective dimension more than~$1$ for which
the class of $u$\+divisible modules coincides with that of
$\G$\+divisible ones.
 In fact, such examples exist already among the maps of localization
by multiplicative subsets $u\:R\rarrow S^{-1}R=U$.
 Consequently, \emph{neither}~\cite[Proposition~6.4]{AHT}
\emph{nor}~\cite[Lemma~1.8(b)]{PMat} hold true for multiplicative
subsets $S\subset R$ containing zero-divisors.
 The following transparent construction of counterexamples was
communicated to the author by M.~Hrbek.

 Let $R$ be a von~Neumann regular commutative ring and $I\subset R$
be an ideal.
 Then $I$ is generated by some set of idempotent elements $e_j\in R$.
 Let $S\subset R$ be the multiplicative subset generated by
the complementary idempotents $f_j=1-e_j$.
 Then $S^{-1}R\simeq R/I$.
 Furthermore, an $R$\+module $M$ is annihilated by~$e_j$ for a given
index~$j$ if and only if $f_j$~acts invertibly in $M$, and if and only
if $f_j$~acts by a surjective endomorphism of~$M$.
 Hence an $R$\+module $M$ is annihilated by the whole ideal $I$ if
and only if all the elements of $S$ act invertibly in $M$, and if and
only if all the elements of $S$ act in $M$ by surjective maps.

 Consequently, all $S$\+divisible $R$\+modules are $S$\+h-divisible
(moreover, all of them are $S^{-1}R$\+modules).
 In other words, if $u\:R\rarrow S^{-1}R=R/I=U$ is the natural
surjective flat epimorphism of rings and $\G$ is the related
perfect Gabriel filter in $R$ (i.~e., the filter of all ideals
intersecting~$S$), then the classes of $u$\+divisible and
$\G$\+divisible $R$\+modules coincide.
 (Cf.\ the discussion of silting and cosilting classes over
von~Neumann regular commutative rings
in~\cite[Section~1.5.3]{Hrb-thesis}.)

 On the other hand, it is well-known that there exist von~Neumann
regular commutative rings of arbitrary homological
dimension (see, e.~g.,~\cite{Pier}).
 So one can choose $R$ and $I$ so as to make the projective dimension
of the $R$\+module $R/I$ to be equal to any chosen nonnegative integer
or infinity.
\end{rem}

\begin{cor} \label{divisible-cor}
 Assuming Setup~\ref{thesetup}, an $R$\+module is $u$\+divisible if
and only if it is\/ $\G$\+divisible.
\end{cor}

\begin{proof}[First proof]
 Follows immediately from Theorems~\ref{projdim1-theorem}
and~\ref{divisible-theorem}.
\end{proof}

 In order to illustrate the workings of contramodule techniques,
in the rest of this section we present an alternative proof of
Corollary~\ref{divisible-cor}.
 For this purpose, we need the following strengthening of
the lemmas from Section~\ref{projdim1-secn}.

\begin{prop} \label{strong-nakayama}
 Let $R$ be a commutative ring and $I=(s_1,\dotsc,s_m)\subset R$ be
a finitely generated ideal.
 Let $B$ be an $R$\+module such that $B/IB=0$ and $C$ be
an $I$\+contramodule $R$\+module.
 Then $\Hom_R(B,C)=0$.
\end{prop}

\begin{proof}
 In other words, the proposition says that an $I$\+contramodule
$R$\+module $C$ has no $R$\+submodules $D$ for which $ID=D$.
 This is provable using infinite summation operations, as hinted
in~\cite[Lemma~4.2]{Pcta}.

 Alternatively, one can use the reflector $\Delta_I\:R\Modl\rarrow
R\Modl_{I\ctra}$ onto the full subcategory of $I$\+contramodule
$R$\+modules $R\Modl_{I\ctra}\subset R\Modl$ (that is, the left
adjoint functor to the inclusion $R\Modl_{I\ctra}\rarrow R\Modl$).
 The functor $\Delta_I$ was constructed in~\cite[Theorem~7.2]{Pcta}.
 Any $R$\+module morphism $B\rarrow C$ into an $I$\+contramodule
$R$\+module $C$ factorizes uniquely as $B\rarrow\Delta_I(B)\rarrow C$,
where $B\rarrow\Delta_I(B)$ is the adjunction morphism.

 The functor $\Delta_I$ is $R$\+linear and right exact, so applying it
to the morphism $(s_1,\dotsc,s_m)\:B^m\rarrow B$, which is surjective
by the assumption $IB=B$, produces a surjective morphism
$(s_1,\dotsc,s_m)\:\Delta_I(B)^m\rarrow\Delta_I(B)$.
 Hence $B/IB=0$ implies $\Delta_I(B)/I\Delta_I(B)=0$.
 But $\Delta_I(B)$ is an $I$\+contra\-module, so by
Lemma~\ref{fin-gen-ideal-weak-nakayama} it follows that
that $\Delta_I(B)=0$.
 Thus any $R$\+module morphism $B\rarrow C$ vanishes.
\end{proof}

\begin{proof}[Second proof of Corollary~\ref{divisible-cor}]
 For any $R$\+module $B$, there is a natural 5\+term exact sequence of
$R$\+modules
\begin{multline} \tag{$*$}
 0\lrarrow\Ext^0_R(K_u^\bu,B)\lrarrow\Hom_R(U,B)\lrarrow B \\
 \lrarrow\Ext^1_R(K_u^\bu,B)\lrarrow\Ext^1_R(U,B)\lrarrow 0
\end{multline}
produced by applying the functor $\Hom_{\sD(R\Modl)}({-},B)$ to
the distinguished triangle
$$
 R\lrarrow U\lrarrow K_u^\bu\lrarrow R[1]
$$
in $\sD(R\Modl)$ (cf.~\cite[formula~($**$) in Section~4]{BP},
\cite[formula~(8.2)]{Pcoun}, or~\cite[formula~(9)]{BP2}).
 By Corollary~\ref{ext-from-K-decomp}, we have an isomorphism
of $R$\+modules $\Ext^1_R(K_u^\bu,B)\simeq\prod_{\m\in\G\cap\Max R}
C(\m)$, where $C(\m)$ are some $\m$\+contramodule $R$\+modules.

 Now assume that the $R$\+module $B$ is $\G$\+divisible.
 Given a maximal ideal $\m\in\G\cap\Max R$, choose a finitely generated
ideal $I\in\G$ such that $I\subset\m$.
 Then $C(\m)$ is an $I$\+contramodule $R$\+module and $B/IB=0$.
 By Proposition~\ref{strong-nakayama}, it follows that
$\Hom_R(B,C(\m))=0$.
 As this holds for all $\m\in\G\cap\Max R$, we can conclude that
the map $B\rarrow\Ext^1_R(K_u^\bu,B)$ vanishes.
 Hence the map $\Hom_R(U,B)\rarrow B$ is surjective and the $R$\+module
$B$ is $u$\+divisible.
\end{proof}

\Section{$u$-Contramodule $R$-Modules}

 In this section we obtain a description of the Geigle--Lenzing
perpendicular subcategory $R\Modl_{u\ctra}=U^{\perp_{0,1}}$ in
the category of $R$\+modules $R\Modl$.

 An $R$\+module $C$ is said to be
a \emph{$u$\+contramodule}~\cite[Section~1]{BP2} if
$C\in U^{\perp_{0,1}}$, that is $\Hom_R(U,C)=0=\Ext^1_R(U,C)$.
 Assuming Setup~\ref{thesetup}, the $R$\+module $U$ has projective
dimension at most~$1$ by Theorem~\ref{projdim1-theorem}.
 By~\cite[Proposition~1.1]{GL} or~\cite[Theorem~1.2(a)]{Pcta}, it
follows that the full subcategory of $u$\+contramodule $R$\+modules
$R\Modl_{u\ctra}$ is closed under the kernels, cokernels, extensions,
and infinite products in $R\Modl$.

 The functor $\Delta_u=\Ext^1_R(K_u^\bu,{-})$ plays an important role
as the reflector onto the full subategory $R\Modl_{u\ctra}\subset
R\Modl$ \,\cite[Proposition~3.2(b)]{BP2}.

\begin{lem} \label{R/I-mod-u-contra}
 Let $u\:R\rarrow U$ be a flat epimorphism of commutative rings and
$I\subset R$ be an ideal such that $IU=U$.
 Then any $R/I$\+module is a $u$\+contramodule.
\end{lem}

\begin{proof}
 Since $U$ is a flat $R$\+module, by~\cite[Lemma~4.1(a)]{PSl} we have
$\Ext^i_R(U,D)\simeq\Ext^i_{R/I}(U/IU,D)=0$ for any $R/I$\+module $D$
and all $i\ge0$.
\end{proof}

\begin{prop} \label{I-contra-u-contra}
 Let $R\rarrow U$ be a flat epimorphism of commutative rings, and
let\/ $\G$ be the related perfect Gabriel filter of ideals in~$R$.
 Assume that the projective dimension of the $R$\+module $U$ does not
exceed~$1$.
 Then, for any ideal $I\in\G$, any $I$\+contramodule $R$\+module is
a $u$\+contramodule.
\end{prop}

\begin{proof}
 Without loss of generality we can assume the ideal $I$ to be finitely
generated.
 Let $C$ be an $I$\+contramodule $R$\+module.
 Then $\Hom_R(U,C)=0$ by Proposition~\ref{strong-nakayama}.

 There are several ways to prove that $\Ext^1_R(U,C)=0$.
 One can observe that any $I$\+contramodule $R$\+module is obtainable
as an extension of two \emph{quotseparated} $I$\+contramodule
$R$\+modules~\cite[Section~5.5]{PSl}, \cite[Section~1]{Pdc};
any quotseparated $I$\+contramodule $R$\+module is the cokernel of
an injective morphism between two $I$\+adically separated and complete
modules; and finally any $I$\+adically separated and complete module is
the kernel of a morphism of $R$\+modules of the form
$\Hom_\boZ(N,\boQ/\boZ)$, where $N$ ranges over
the $\G$\+torsion $R$\+modules~\cite[Proposition~5.6]{Pcoun}.
 Since $R$\+modules of the latter form belong to $R\Modl_{u\ctra}$,
it follows that all the $I$\+contramodule $R$\+modules do.

 Alternatively, one can say that all $I$\+contramodule $R$\+modules $C$
are \emph{simply right obtainable} from $R/I$\+modules (i.~e.,
obtainable using the passages to extensions, cokernels of monomorphisms,
infinite products, and infinitely iterated extensions in the sense of
the projective limit~\cite[Section~3]{PSl}, \cite[Section~2]{PSl2}).
 This is the assertion of~\cite[Lemma~8.2]{PSl} (see also~\cite[proof of
Theorem~9.5]{Pcta}).
 Since $\Ext^i_R(U,D)=0$ for any $R/I$\+module $D$ and all $i>0$
by Lemma~\ref{R/I-mod-u-contra}, it follows that
$\Ext^i_R(U,C)=0$ for $i>0$ by~\cite[Lemma~3.4]{PSl}.
\end{proof}

\begin{thm} \label{u-contramodules-described}
 Assuming Setup~\ref{thesetup}, an $R$\+module $C$ is
a $u$\+contramodule if and only if it is isomorphic to a product of
some\/ $\m$\+contramodule $R$\+modules $C(\m)$ over the maximal ideals\/
$\m\in\G\cap\Max R$,
$$
 C\,\simeq\,\prod_{\m\in\G\cap\Max R} C(\m).
$$
 Such an infinite product decomposition is unique and functorial when
it exists, and the $R$\+modules $C(\m)$ can be recovered as
the colocalizations $C(\m)=\Hom_R(R_\m,C)$.
\end{thm}

\begin{proof}
 The assertion ``if'' is provided by
Proposition~\ref{I-contra-u-contra}.
 To prove the ``only if'', notice that the natural morphism
$C\rarrow\Ext^1_R(K_u^\bu,C)$ is an isomorphism for any
$u$\+contramodule $R$\+module $C$ (in view of the exact
sequence~($*$) from Section~\ref{divisible-secn}).
 Now Corollary~\ref{ext-from-K-decomp} provides the desired direct
product decomposition.

 Let us compute the colocalizations.
 Any $\m$\+contramodule $R$\+module $C(\m)$ is an $R_\m$\+module by
Lemma~\ref{maximal-ideal-contramodule}, hence
$\Hom_R(R_\m,C(\m))=C(\m)$.
 On the other hand, for any maximal ideal $\n\ne\m$ of the ring $R$,
we have $\Hom_R(R_\n,C(\m))=0$ by
Lemma~\ref{ext-from-invertible-to-contra} (choose any element
$t\in\m\setminus\n$).
 Finally, our direct product decomposition is unique and functorial by
Proposition~\ref{ext-between-products} (for $i=0$).
\end{proof}

 Thus, assuming Setup~\ref{thesetup}, the category of $u$\+contramodule
$R$\+modules is equivalent to the Cartesian product of the categories
of $\m$\+contramodule $R$\+modules, taken over the maximal ideals~$\m$
of the ring $R$ belonging to the filter~$\G$.
 The product functor
$$
 (C(\m))_{\m\in\G\cap\Max R}\,\longmapsto\,\prod_{\m\in\G\cap\Max R}
 C(\m)
$$
establishes the equivalence, with an inverse equivalence provided by
the colocalization functor
$$
 C\,\longmapsto\,(C^\m=\Hom_R(R_\m,C))_{\m\in\G\cap\Max R}.
$$

\Section{$U$-Strongly Flat $R$-Modules}

 We refer to~\cite[beginning of Section~9]{Pcoun} for a general
discussion of $U$\+weakly cotorsion and $U$\+strongly flat left
$R$\+modules for a left flat epimorphism of associative rings
$u\:R\rarrow U$.
 In this section, as in the rest of this paper, we restrict ourselves
to the commutative case.

 Let $u\:R\rarrow U$ be a flat epimorphism of commutative rings.
 A left $R$\+module $C$ is said to be \emph{$U$\+weakly cotorsion}
if $\Ext^1_R(U,C)=0$.
 A left $R$\+module $F$ is said to be \emph{$U$\+strongly flat} if
$\Ext^1_R(F,C)=0$ for all $U$\+weakly cotorsion $R$\+modules~$C$.

 By~\cite[Theorem~4.4]{GL}, we have $\Ext^i_R(U,U^{(X)})\simeq
\Ext^i_U(U,U^{(X)})=0$ for all $i>0$ and any set $X$;
in particular, $\Ext^1_R(U,U^{(X)})=0$ (where $U^{(X)}$ denotes
the direct sum of $X$ copies of~$U$).
 Hence~\cite[Corollary~6.13]{GT} provides the following description
of $U$\+strongly flat $R$\+modules.
 An $R$\+module $F$ is $U$\+strongly flat if and only if it is
a direct summand of an $R$\+module $G$ appearing in a short exact
sequence of $R$\+modules
$$
 0\lrarrow V\lrarrow G\lrarrow W\lrarrow 0,
$$
where $V$ is a free $R$\+module and $W$ is a free $U$\+module.

 The notion of a \emph{simply right obtainable} module (from a given
class of modules) was already mentioned in the proof of
Proposition~\ref{I-contra-u-contra}.
 We refer to~\cite[Section~3]{PSl} or~\cite[Section~2]{PSl2} for
a detailed discussion.

\begin{prop} \label{u-contramodules-obtainable}
 Assuming Setup~\ref{thesetup}, the class of all $u$\+contramodule
$R$\+modules $R\Modl_{u\ctra}\subset R\Modl$ coincides with the class
of all $R$\+modules simply right obtainable from $R/I$\+modules, where
$I$ ranges over the Gabriel filter\/~$\G$.
\end{prop}

\begin{proof}
 By Theorem~\ref{u-contramodules-described}, any $u$\+contramodule
$R$\+module is obtainable as a product of $\m$\+contramodules over
the maximal ideals $\m\in\G\cap\Max R$.
 For any~$\m$, there exists a finitely generated ideal $I\subset\m$
such that $I\in\G$.
 Then any $\m$\+contramodule is an $I$\+contramodule.
 Finally, all $I$\+contramodule $R$\+modules are simply right
obtainable from $R/I$\+modules by~\cite[Lemma~8.2]{PSl2}
or~\cite[proof of Theorem~9.5]{Pcta}.

 Conversely, since by Theorem~\ref{projdim1-theorem} we have
$\pd_RU\le1$, the class of all $u$\+contramodule $R$\+modules is closed
under the kernels, cokernels, extensions, and infinite products in
$R\Modl$ by~\cite[Proposition~1.1]{GL} or~\cite[Theorem~1.2(a)]{Pcta}.
 Since all the $R/I$\+modules with $I\in\G$ belong to $R\Modl_{u\ctra}$
by Lemma~\ref{R/I-mod-u-contra}, it follows that all the $R$\+modules
simply right obtainable from $R/I$\+modules, $I\in\G$, belong to
$R\Modl_{u\ctra}$.
\end{proof}

\begin{prop} \label{weakly-cotorsion-obtainable}
 Assuming Setup~\ref{thesetup}, the class of all $U$\+weakly cotorsion
$R$\+modules coincides with the class of all $R$\+modules simply right
obtainable from $U$\+modules and $R/I$\+modules with $I\in\G$.
\end{prop}

\begin{proof}
 For any $R$\+module $B$, it is clear from
Corollary~\ref{ext-from-K-decomp} and
Proposition~\ref{I-contra-u-contra} that both the $R$\+modules
$\Ext_R^0(K_u^\bu,B)$ and $\Ext_R^1(K_u^\bu,B)$ are
$u$\+contramodules (cf.~\cite[Lemma~2.6(a) and~(c)]{BP2}).
 Now it follows from the exact sequence~($*$) from
Section~\ref{divisible-secn} that any $U$\+weakly cotorsion
$R$\+module $C$ is obtainable as an extension of a $u$\+contramodule
$\Ext^1_R(K_u^\bu,C)$ and the cokernel of an injective morphism from
a $u$\+contramodule to a $U$\+module $\Ext^0_R(K_u^\bu,C)\rarrow
\Hom_R(U,C)$.
 It remains to use Proposition~\ref{u-contramodules-obtainable} for
obtainability of $u$\+contramodules.

 Conversely, all $R/I$\+modules are $U$\+weakly cotorsion $R$\+modules
by Lemma~\ref{R/I-mod-u-contra}, and all $U$\+modules $D$ are
$U$\+weakly cotorsion $R$\+modules, since $\Ext^i_R(U,D)\simeq
\Ext^i_U(U,D)=0$ for $i>0$ by~\cite[Theorem~4.4]{GL}.
 By virtue of~\cite[Lemma~3.4]{PSl}, it follows that all
the $R$\+modules simply right obtainable from $U$\+modules and
$R/I$\+modules are $U$\+weakly cotorsion.
\end{proof}

 The following theorem can be viewed as confirming a version
of~\cite[Optimistic Conjecture~1.1]{PSl2} for flat epimorphisms
of commutative rings (cf.~\cite[Corollary~9.7]{Pcoun}).
 It is also a generalization of~\cite[Proposition~7.13]{BP}.

\begin{thm} \label{strongly-flat-theorem}
 Assume Setup~\ref{thesetup}, and let $F$ be a flat $R$\+module.
 Then $F$ is $U$\+strongly flat if and only if it satisfies
the following two conditions:
\begin{enumerate}
\renewcommand{\theenumi}{\roman{enumi}}
\item the $U$\+module $U\ot_RF$ is projective;
\item for every ideal $I\in\G$, the $R/I$\+module $F/IF$ is
projective.
\end{enumerate}
\end{thm}

\begin{proof}
 The ``only if'' assertion holds for any flat epimorphism of
commutative rings $u\:R\rarrow U$ and the related perfect Gabriel
filter~$\G$.
 It is clear from the description of $U$\+strongly flat $R$\+modules
$F$ as the direct summands of the $R$\+modules $G$ appearing in
short exact sequences of $R$\+modules $0\rarrow V\rarrow G\rarrow W
\rarrow0$ with a free $R$\+module $V$ and a free $U$\+module $W$
that conditions~(i\+-ii) are satisfied for~$F$.

 Our proof of the ``if'' depends on the assumption of
Setup~\ref{thesetup}.
 Let $F$ be a flat $R$\+module satisfying~(i\+-ii), and let $C$
be a $U$\+weakly cotorsion $R$\+module.
 We need to show that $\Ext^1_R(F,C)=0$, but it will be more convenient
for us to prove that $\Ext^i_R(F,C)=0$ for all $i>0$.
 By Proposition~\ref{weakly-cotorsion-obtainable}, \,$C$~is simply right
obtainable from $U$\+modules and $R/I$\+modules with $I\in\G$.
 According to~\cite[Lemma~3.4]{PSl}, it suffices to consider the cases
when $C$ is either a $U$\+module or an $R/I$\+module.

 For any $U$\+module $D$, we have $\Ext^i_R(F,D)\simeq
\Ext^i_U(U\ot_RF,\>D)$ by~\cite[Lemma~4.1(a)]{PSl}.
 In view of condition~(i), \,$\Ext^i_U(U\ot_RF,\>D)=0$.
 Similarly, for any $R/I$\+module $D$, we have $\Ext^i_R(F,D)\simeq
\Ext^i_{R/I}(F/IF,D)$ by~\cite[Lemma~4.1(a)]{PSl}.
 In view of condition~(ii), \,$\Ext^i_{R/I}(F/IF,D)=0$.
\end{proof}

 Let $u\:R\rarrow U$ be a flat epimorphism of commutative rings
and $\G$ be the related perfect Gabriel filter of ideals in~$R$.
 Following~\cite[Sections~6 and~8]{BG2}, we will say that the ring
$R$ is \emph{$\G$\+almost perfect} if the ring $U=R_\G$ is perfect
and the rings $R/I$ are perfect for all $I\in\G$.
 Clearly, the ring epimorphism $u\:R\rarrow U$ satisfies
Setup~\ref{thesetup} whenever $R$ is $\G$\+almost perfect for
the related Gabriel filter~$\G$ (as all perfect commutative rings
are semilocal of Krull dimension zero).

\begin{cor} \label{almost-perfect-cor}
 Let\/ $\G$ be the perfect Gabriel filter of ideals related to
a flat epimorphism of commutative rings $u\:R\rarrow U$.
 Then the ring $R$ is\/ $\G$\+almost perfect if and only if all
the flat $R$\+modules are $U$\+strongly flat.
\end{cor}

\begin{proof}
 The ``if'' assertion is provable similarly to~\cite[Lemma~7.8]{BP}.
 One shows that all flat $U$\+modules are projective whenever all
flat $R$\+modules are $U$\+strongly flat, and all Bass flat
$R/I$\+modules are projective whenever all Bass flat $R$\+modules
are $U$\+strongly flat.
 We skip the (straightforward) details.

 The implication ``only if'' is a simple corollary of
Theorem~\ref{strongly-flat-theorem}.
 Assume that the ring $R$ is $\G$\+almost perfect; then
Setup~\ref{thesetup} is satisfied.
 Let $F$ be a flat $R$\+module.
 Then the $U$\+module $U\ot_RF$ is flat, and the $R/J$\+module
$F/JF$ is flat for every ideal $J\subset R$.
 Since the ring $U$ is perfect by assumption, it follows that
the $U$\+module $U\ot_RF$ is projective.
 Since the ring $R/I$ is perfect for all $I\in\G$, it follows
that the $R/I$\+module $F/IF$ is projective, too.
 Thus conditions~(i\+-ii) hold, and the $R$\+module $F$ is
$U$\+strongly flat by Theorem~\ref{strongly-flat-theorem}.
\end{proof}

\end{document}